\documentclass[11pt]{article}
\usepackage[T1]{fontenc}
\usepackage{lmodern,amsmath,amssymb,amsthm,mathtools,comment}
\usepackage[margin=1in]{geometry}
\usepackage{microtype,booktabs}
\usepackage[colorlinks=true,linkcolor=blue,citecolor=blue,urlcolor=blue]{hyperref}
\usepackage{enumitem}
\setlist{itemsep=3pt,topsep=5pt}
\newtheorem{theorem}{Theorem}[section]
\newtheorem{proposition}[theorem]{Proposition}
\newtheorem{lemma}[theorem]{Lemma}
\newtheorem{corollary}[theorem]{Corollary}
\theoremstyle{definition}

\newtheorem*{dfn}{Definition}
\theoremstyle{remark}
\newtheorem{remark}[theorem]{Remark}
\newcommand{\Z}{\mathbb Z}
\newcommand{\Q}{\mathbb Q}
\newcommand{\C}{\mathbb C}
\newcommand{\F}{\mathbb F}
\newcommand{\HH}{\mathbb H}
\newcommand{\spt}{\operatorname{spt}}

\newcommand{\End}{\operatorname{End}}

\newcommand{\covol}{\operatorname{covol}}
\newcommand{\reg}{\mathrm{reg}}

\newcommand{\leg}[2]{\left(\frac{#1}{#2}\right)}
\newcommand{\chip}{\chi_{12}}
\newcommand{\e}{\mathrm e}
\numberwithin{equation}{section}
\title{Class polynomial obstructions to canonical congruences\\for the smallest parts function}
\author{Larry Rolen}
\begin{document}
\maketitle
\begin{abstract}
We resolve an open conjecture by proving that the only ``canonical'' Ramanujan congruence primes for the Andrews' smallest parts function, $\spt(n)$, are $\ell= 5, 7, 13$. 
Building on work of Ahlgren and Kim, we show that any such congruence forces square-index coefficients of certain weakly holomorphic
forms of weight $1/2$ to vanish modulo the congruence prime. A normalized
Borcherds product converts this into the vanishing of the derivative of a
quotient of two oriented CM polynomials on $X_0(6)$. A quaternionic
volume estimate in an Eichler order of level $6$ prevents that vanishing
when a suitable inert discriminant is small. Trevi\~no's 
character estimate supplies such a discriminant for every prime at least
$10^{11}$. The remaining range is covered by the previous calculations
of Ahlgren and Kim.
These results have bee verified in Lean by AxiomProver.
\end{abstract}

\section{Statement and outline}
Let $\spt(n)$ denote the total number of occurrences of a smallest part
in all partitions of $n$.
This function was introduced by Andrews in \cite{Andrews}
who  used rank and crank moments for integer partitions to show the following Ramanuan-style congruences. Specifically, he showed that for all $n$, we have
\[
\spt(5n+4)\equiv0\pmod5,\qquad
\spt(7n+5)\equiv0\pmod7,\qquad
\spt(13n+6)\equiv0\pmod{13}.
\]
Many further works have studied the $\spt$ function and further generalized congruences, crank-type explanations, and arithmetic; see 
\cite{AA,ABL,MockModGrids,AGL,Garvan,Ono}
for a few such papers.
Since Andrews' discovery of these beautiful congruences, it has been a long-standing conjecture in the community that $5,7,$ and $13$ are the only primes which admit congruences of this shape. To be precise, we introduce the following notion, where $\overline{a}\pmod \ell$ denotes an inverse of $a$ modulo $\ell$. 
\begin{dfn}
For a prime $\ell\ge5$, let $b_\ell\in\{1,\ldots, \ell-1\}$ be the reduced representative of 
$\overline{24}\pmod \ell$. We say that $\ell$ is a \emph{canonical congruence
prime} if
\begin{equation}\label{eq:canonical}
\spt(\ell n+b_\ell)\equiv0\pmod \ell\qquad \text{ for all } n\ge0.
\end{equation}
\end{dfn}
In this language, Ahlgren and Kim \cite{AK} studied such canonical congruences and showed that it can only exist when a computable cusp form of level $1$ and weight $\ell+1$ vanishes. The role of $\ell\in\{5, 7, 13\}$ then becomes that the spaces $S_6, S_8, $ and $S_{14}$ are all zero-dimensional. They then used this criterion to numerically show that 
there are no additional canonical congruence
primes $\ell<10^{11}$. Here, we build on their techniques to show that Andrews' congruences are indeed the only ones of this type. 

\begin{theorem}\label{thm:main}
The canonical congruence primes for $\spt(n)$ are precisely $5$, $7$, and $13$.
\end{theorem}

\section*{Use of generative AI and autoformalization}
The proofs arose out of a back and forth conversation between the author with AI. The author takes responsibility for their correctness and for the clarity of the exposition. 
The author is also grateful to Ken Ono for his assistance in autoformalizing the results of this paper.
After completing the first draft, the author sent it to Ken Ono, who used AxiomProver to verify the proofs in Lean. These formal verifications were completely carried out by AxiomProver.
The relevant files are all posted in the following repository:
\url{https://github.com/larryrolen/Spt_Noncongruences_Lean}

\section{The mock modular form for $\spt$ and the action of Hecke operators}

As usual, set $q:=\e^{2\pi i\tau}$ and denote differentiation by $\Theta:=q\,\frac{d}{dq}$. Further denote the integer partition function by $p(n)$.
The generating function for $\spt(n)$ is closely related to a mock modular form; see \cite{Bringmann}, Section~6 of \cite{Zagier}, and Section~3 of \cite{AA} . 
To describe this, let
\[
f(\tau):=\sum_{n\equiv23 \pmod{24}}a(n)q^{n/24}
=q^{-1/24}(1-35q-130q^2-273q^3-\cdots),
,
\]
where
\[
a(24n-1):=-12\spt(n)-(24n-1)p(n).\]
We also recall that $\eta(\tau):=q^{1/24}\prod_{n\geq1}(1-q^n)$ is Dedekind's eta function, which transforms as a modular form on $\mathrm{SL}_2(\Z)$
with multiplier $\nu_\eta$. 
Bringmann proved \cite{Bringmann} that $f(\tau)$ has 
a harmonic Maass completion $\widehat f$ with weight $3/2$ and multiplier
$\overline{\nu_\eta}$.
This completion is determined by its shadow:
\begin{equation}\label{eq:shadow}
\xi_{3/2}\widehat f=\frac{3\sqrt 6}{\pi}\eta.
\end{equation}
where $\xi_k g:=2iy^k\overline{\partial g/\partial\bar\tau}$.

As we shall see, we can use Hecke and Hecke-like operators to reduce the study of $\spt$-congruences to questions about (weakly) holomorphic modular forms. Thus, we won't need the full strength of harmonic Maass forms theory in what follows.
We also require the $U_\ell$-operator, which sends the $q$-series 
 $g=\sum_n c(n)q^{n/24}$ to 
$g|U_\ell:=\sum_n c(\ell n)q^{n/24}$. 
Using the definition of $a(n)$, we find that 
\eqref{eq:canonical} is equivalent to $f|U_\ell\equiv0\pmod \ell$.
We use the following consequence of the theta-cycle argument of Ahlgren
and Kim.

\begin{proposition}[Ahlgren--Kim \cite{AK}]\label{prop:theta}
If $\ell$ is a canonical congruence prime, then
\begin{equation}\label{eq:theta}
\Theta^{(\ell-1)/2}f\equiv\varepsilon_\ell f\pmod \ell,
\qquad \text{ where } \varepsilon_\ell=\leg{-24}{\ell}.
\end{equation}
In particular, $a(n)\equiv0\pmod \ell$ whenever $\leg{-n}{\ell}=-1$.
\end{proposition}
\begin{proof}
The first assertion follows from \cite[Proposition~6]{AK}, after cancelling the factor of $\eta$ out.
For the second claim, note that the iterated derivative operator multiplies a coefficient of $q^{n/24}$
by $(n/24)^{(\ell-1)/2}$. An application of Euler's criterion for quadratic residues then completes the proof. 
\end{proof}

A well-known identity of Euler states that $\eta$ can be written using the Kronecker character 
$\chip(m):=\leg{12}{m}$
as the theta series:
\begin{equation}\label{eq:eta-theta}
\eta(\tau)=\sum_{m\ge1}\chip(m)q^{m^2/24}.
\end{equation}
\begin{proposition}\label{Grid}
For each $m>1$ with $(m,6)=1$ there is a unique integral weakly
holomorphic form
\begin{equation}\label{eq:Gm}
G_m=q^{-m^2/24}-\chip(m)q^{-1/24}
       +\sum_{n>0}b_m(n)q^{n/24}
\in M^!_{3/2}(\overline{\nu_\eta}).
\end{equation}
\end{proposition}
\begin{remark}
We extend the definition by setting $G_1:=0$.
\end{remark}
\begin{proof}
We first show that such a weakly holomorphic modular form exists. 
Using the modular $j$-invariant and the usual normalized Eisenstein series $E_k=1+O(q)\in M_k$, note that 
\[
\mathcal G:=-\frac{\Theta j}{\eta}
=\frac{E_4^2E_6}{\eta^{25}}
=q^{-25/24}\left(1+O(q)\right)\in M^!_{3/2}(\overline{\nu_\eta}).
\]
has integral coefficients. 
This form lives in the right modular space, but we need to adjust its principal part. Recall that polynomials in $j$ yield weakly holomorphic modular functions of level $1$ with 
arbitrary principal part. For any $m\geq1$ we obtain a claimed form $G_m$ by 
multiplying $\mathcal G$ by an appropriate polynomial in $j$ to force the desired principal part. 

Uniqueness follows from the fact that  the difference of two such forms,
multiplied by $\eta$, would belong to $M_2(\mathrm{SL}_2(\Z))=\{0\}$.
\end{proof}

For a prime $\ell\ge5$ we use the weight $3/2$ Hecke operator:
\begin{equation}\label{eq:Hecke}
\left(\sum_n c(n)q^{n/24}\right)\big|T(\ell^2)
:=\sum_n\left(c(\ell^2n)+\chip(\ell)\leg{-n}{\ell}c(n)
                   +\ell c(n/\ell^2)\right)q^{n/24}.
\end{equation}
We now need some Hecke relations; see also 
\cite{MockModGrids,Ono} for very similar discussions.

\begin{lemma}\label{lem:hecke-identities}
For every prime $\ell\ge5$, we have
\begin{equation}\label{eq:f-Hecke}
f|T(\ell^2)-\chip(\ell)(1+\ell)f=\ell G_\ell.
\end{equation}
For $m>1$, $(m,6)=1$, we also have
\begin{align}
G_m|T(\ell^2)={}&\ell G_{m\ell}
 +\chip(\ell)\leg{m^2}{\ell}G_m
 +G_{m/\ell}
 -\ell\chip(m)G_\ell\label{eq:G-Hecke},
\end{align}
where $G_{m/\ell}$ is considered to be $0$ if $\ell\nmid m$. 
\end{lemma}
\begin{proof}
We require the weight $1/2$ version of this Hecke operator, which acts on Fourier coefficients by  sending
\[
c(n)\mapsto c(\ell^2n)+\ell^{-1}\chip(\ell)\leg{n}{\ell}c(n)
                         +\ell^{-1}c(n/\ell^2).
\]
Equation~\eqref{eq:eta-theta} directly gives
$\eta|T_{1/2}(\ell^2)=\chip(\ell)(1+\ell^{-1})\eta$.
It is well-known (see, e.g., (7.5) of \cite{HMFBook})  
that the shadow operator  $\xi_{3/2}$ intertwines with Hecke operators in the dual weights $1/2$ and $3/2$.
up to a factor of
$\ell$.
Thus,  the shadow of
$\widehat f|T(\ell^2)-\chip(\ell)(1+\ell)\widehat f$ vanishes.
This precisely says that this difference is weakly holomorphic. 
Moreover, its principal part is easily computed to be
$\ell q^{-\ell^2/24}-\ell\chip(\ell)q^{-1/24}$.
The uniqueness property of Proposition~\ref{Grid} proves \eqref{eq:f-Hecke}.

To show \eqref{eq:G-Hecke}, first apply the weight $3/2$ Hecke operator in \eqref{eq:Hecke} to the 
two terms in the principal part of $G_m$. 
The new principal part matches the right hand side of the claimed identity. 
The same uniqueness argument finishes the proof. 
Note that the convention $G_1=0$
accounts for the cancellation when $m=\ell$.
\end{proof}
We are now in position to reduce the problem of studying canonical congruence primes to question purely of weakly holomorphic, as opposed to mock or harmonic Maass, forms.

\begin{proposition}\label{prop:G-support}
If $\ell$ is a canonical congruence prime, then for every $m$ coprime to $6\ell$, we have
\begin{equation}\label{eq:G-theta}
\Theta^{(\ell-1)/2}G_m\equiv\varepsilon_\ell G_m\pmod \ell.
\end{equation}
Thus, $b_m(n)\equiv0\pmod\ell $ if $\leg{-n}{\ell}=-1$.
\end{proposition}
\begin{proof}
For primes $p\ne \ell$, the Hecke operator in \eqref{eq:Hecke} commutes modulo $\ell$
with $\Theta^{(\ell-1)/2}$.
To see why, note that multiplication or division of an index by
$p^2$ does not change its quadratic character modulo $\ell$, and that the
claim is also valid at indices divisible by $\ell$.
For $m=p\geq 5$, using \eqref{eq:theta} and \eqref{eq:f-Hecke} and dividing by $p$ 
shows \eqref{eq:G-theta}.
To show  \eqref{eq:G-theta} for any $m$ coprime to $6\ell$, induct on the number of prime factors of $m$ (counted with multiplicity, say). 
The final assertion follows by the same
coefficient comparison as in Proposition~\ref{prop:theta}.
\end{proof}

\section{Coefficients and inner products of auxiliary weakly holomorphic modular forms}

Fix $n=24k-1>0$. Triangular elimination on the set of weakly holomorphic modular forms
$\{\eta,\eta j,\ldots,\eta j^k\}$ yields a unique monic polynomial
$P_n\in\Z[X]$ of degree $k$ such that
\begin{equation}\label{eq:Hn}
H_n(\tau):=\eta(\tau)P_n(j(\tau))
=q^{-n/24}-a(n)q^{1/24}+O(q^{25/24}).
\end{equation}
Write $H_n=:\sum_r h_n(r)q^{r/24}$. In particular, we have $h_n(1)=-a(n)$.

\begin{lemma}\label{lem:duality}
For $m>1$ with $(m,6)=1$, we have the coefficient duality
\begin{equation}\label{eq:duality}
h_n(m^2)=-b_m(n)-\chip(m)a(n).
\end{equation}
If $\ell$ is a canonical congruence prime and $\leg{-n}{\ell}=-1$, then for any $m$ coprime to $6\ell$ we have
\begin{equation}\label{eq:square-vanishing}
h_n(m^2)\equiv0\pmod \ell.
\end{equation}
\end{lemma}
\begin{proof}
The product $H_nG_m$ is a weakly holomorphic modular form of weight two
on $\mathrm{SL}_2(\Z)$, so its constant term is zero. 
Plugging in the principal parts then gives
\[
0=[q^0](H_nG_m)=b_m(n)+h_n(m^2)-\chip(m)h_n(1),
\]
proving \eqref{eq:duality}. Propositions~\ref{prop:theta} and
\ref{prop:G-support} prove \eqref{eq:square-vanishing} for $m>1$. The case when
$m=1$ is the relation $h_n(1)=-a(n)\equiv0\pmod\ell$.
\end{proof}

We now require the regularized Petersson  inner product obtained by integrating over a 
truncated
fundamental domain 
 (see, e.g., Section~5.1 of \cite{HMFBook}), given by
\[
\langle H,\eta\rangle^{\reg}
:=\lim_{T\to\infty}\int_{\mathcal F_T}
       H(\tau)\overline{\eta(\tau)}v^{1/2}\,\frac{du\,dv}{v^2},
\]
where $\tau=:u+iv$ ($u,v\in\mathbb R$) and $\mathcal F_T:=\{\tau\in\mathbb H \colon |u|\leq \frac12, |\tau|\geq1, v\leq T\}$.
This limit exists for the weakly holomorphic modular forms $H$ of weight $1/2$ we need.
A standard Stokes Theorem argument gives the following result.
\begin{lemma}\label{lem:pairing}
If $H\in M^!_{1/2}(\nu_\eta)$, then
\begin{equation}\label{eq:pairing}
\frac{3\sqrt 6}{\pi}\cdot\langle H,\eta\rangle ^{\reg}=[q^0](Hf).
\end{equation}
In particular, $\langle H_n,\eta\rangle^{\reg}=0$.
\end{lemma}
\begin{proof}
From \eqref{eq:shadow}, we have
$\partial_{\bar\tau}\widehat f=-(3\sqrt 6/\pi)\cdot y^{-3/2}\bar\eta/(2i)$.
Since $H$ is holomorphic on $\HH$,
\[
d(H\widehat f\,d\tau)
  =-\frac{3\sqrt 6}{\pi}\,H\bar\eta\,v^{-3/2}\,du\,dv.
\]
The differential $H\widehat f\,d\tau$ is invariant. The paired side
boundaries of $\mathcal F_T$ cancel in Stokes' theorem, and the top
boundary is traversed from right to left. It follows that
\[
\frac{3\sqrt 6}{\pi}\cdot\int_{\mathcal F_T}H\bar\eta\,v^{-3/2}\,du\,dv
=\int_{-1/2}^{1/2}H(u+iT)\widehat f(u+iT)\,du.
\]
The contribution of the holomorphic part is the constant term $[q^0](Hf)$.
The nonholomorphic part has negative-square indices, with incomplete
gamma factors. Only a positive coefficient of $H$ can pair with such an
index to contribute a constant term, and these contributions tend to
zero exponentially. The finitely many polar terms of $H$ have no
constant-term partner in this nonholomorphic part. This proves the
limit and \eqref{eq:pairing}.
Finally, \eqref{eq:Hn} and $a(-1)=1$ give
$[q^0](H_nf)=a(n)-a(n)=0$.
\end{proof}

\section{Class polynomials}

From now on suppose that $n$ is squarefree and $n\equiv23\pmod{24}$.
Then $-n$ is a fundamental discriminant, and both $2$ and $3$ split in the imaginary quadratic field
$K:=\Q(\sqrt{-n})$. For $r\in\{1,5,7,11\}$ let
$\mathcal Q_{n,r}$ be a set of representatives of $\Gamma_0(6)$-classes of primitive
positive definite integral forms $[a,b,c]$ satisfying
\[
b^2-4ac=-n,\qquad 6\mid a,\qquad b\equiv r\pmod{12}.
\]
The corresponding CM points are $\tau_Q:=(-b+i\sqrt n)/(2a)$, and the corresponding Heegner divisor is
$D_{n,r}:=\sum_{Q\in\mathcal Q_{n,r}}[\tau_Q]$.
Each set $Q_{n,r}$ has $h(-n)$ classes.
The reason we choose the four values of $r$ arises from picking choices of primes above $2$ and $3$.

Consider the modular function
\begin{equation}\label{eq:t}
t(\tau):=\frac{\eta(\tau)^5\eta(3\tau)}{\eta(2\tau)\eta(6\tau)^5}
=q^{-1}-5+O(q).
\end{equation}
Define the class polynomials
\begin{align}
A_n(X)&:=\prod_{r\in\{1,11\}}\prod_{Q\in\mathcal Q_{n,r}}
                   (X-t(\tau_Q)),\label{eq:A}\\
B_n(X)&:=\prod_{r\in\{5,7\}}\prod_{Q\in\mathcal Q_{n,r}}
                   (X-t(\tau_Q)).\label{eq:B}
\end{align}

We first show the following properties.
\begin{lemma}\label{lem:integral-coordinate}
The function $t$ identifies $X_0(6)_{\Z[1/6]}$ with
$\mathbb P^1_{\Z[1/6]}$. Moreover,
$A_n,B_n\in\Z[1/6][X]$, and both are monic of degree $2h(-n)$.
\end{lemma}
\begin{proof}
As $t$ is an eta quotient, following Ligozat~\cite{Ligozat}, 
we can work out its orders at the cusps. 
We find that $t$ has weight $0$ and 
trivial character.
Indeed, for the exponents $(5,-1,1,-5)$ indexed by $1,2,3,6$, the two
congruence sums are $-24$ and $24$, and the character parameter is
$2^{-6}3^{-4}$, a perfect square. The cusp orders at $0,1/2,1/3,\infty$ are,
respectively $1,0,0,-1$. The eta function has no zeros on $\HH$.
Thus, on the generic fiber of $X_0(6)_{\Z[1/6]}$,
$\operatorname{div}(t)=[0]-[\infty]$, and $t$ has degree one.

Here is also the integral model justification needed for specialization.
Away from $6$, the compactified modular curve is smooth with geometrically
connected fibers, and its four cusp sections are disjoint. The fibers
are smooth curves of genus zero, and hence irreducible. A possible vertical
component of $\operatorname{div}(t)$ would be an entire fiber. It has
multiplicity zero, as seen in the formal neighborhood of $\infty$ since 
$t=q^{-1}u(q)$ for $u\in1+q\Z[[q]]$. This shows that $\operatorname{div}(t)=[0]-[\infty]$ on the model over $\Z[1/6]$. The sections $1,t$ of
$\mathcal O([\infty])$ have no common zero and give a morphism to
$\mathbb P^1$ of degree one on every geometric fiber. This is an
isomorphism.

The CM ideal-class action over $K$ preserves each orientation $r$ and 
complex conjugation changes $r$ to $-r$. Thus, both polynomials 
$A_n(X)$ and $B_n(X)$ have rational coefficients. At a prime above
any $\ell>3$, a CM elliptic curve has potentially good reduction, and its
cyclic subgroup of order $6$ extends after a finite extension. The CM
point therefore specializes to the interior of this integral model.
Since $t$ has its only pole at the cusp $\infty$, its CM values are
integral at all such primes. Rationality proves the claimed coefficient
ring. The degree follows from directly from  $|Q_{n,r}|=h(-n)$.
\end{proof}
Using Bruinier-Ono's theory of Borcherds products \cite{BO}, we next show the following.
\begin{theorem}\label{thm:product}
For every squarefree $n\equiv23\pmod{24}$,
\begin{equation}\label{eq:product}
\frac{A_n(t(\tau))}{B_n(t(\tau))}
=\prod_{m\ge1}(1-q^m)^{\chip(m)h_n(m^2)}.
\end{equation}
The identity holds analytically near $q=0$, and as an identity of
formal power series with coefficients integral away from $6$.
\end{theorem}
\begin{proof}
We give the representation, cusp, and divisor normalizations explicitly.
Let $D=:\Z/12\Z$ with quadratic form $Q(r):=r^2/24$. For its standard
basis $\mathfrak e_r$, the Weil representation $\rho_6$ satisfies
\[
\rho_6(T)\mathfrak e_r=\e(r^2/24)\mathfrak e_r,\qquad
\rho_6(S)\mathfrak e_r=\frac{\e(-1/8)}{\sqrt{12}}
       \sum_{s\pmod{12}}\e(-rs/12)\mathfrak e_s,
\]
where $\e(x):=e^{2\pi ix}$. Put
\[
v:=\mathfrak e_1-\mathfrak e_5-\mathfrak e_7+\mathfrak e_{11}
 =\sum_{r\pmod{12}}\chip(r)\mathfrak e_r.
\]
Using the quadratic Gauss sum
$\sum_{r\pmod{12}}\chip(r)\e(-rs/12)=\sqrt{12}\chip(s)$
we find that $\mathbf H_n:=H_nv$ belongs to $M^!_{1/2,\rho_6}$.
It has constant term zero and 
principal part $q^{-n/24}v$.

Consider the theta function
\[
\theta(\tau):=\sum_{r\pmod{12}}\sum_{b\equiv r\pmod{12})}q^{b^2/24}
                           \mathfrak e_r.
\]
Pointwise contraction with $v$ gives 
\[
\theta_1(\tau)-\theta_5(\tau)-\theta_7(\tau)+\theta_{11}(\tau)
=\sum_{b\in\mathbb Z}\chi_{12}(b)q^{b^2/24}
=
2\eta(\tau)\] 
by
\eqref{eq:eta-theta}. Define Atkin--Lehner permutations $\sigma_c$, for
$c\mid6$, by
\[
\sigma_c(r)\equiv-r\pmod{2c},\qquad
\sigma_c(r)\equiv r\pmod{12/c}.
\]
The vector $v$ is an eigenvector of $\sigma_c$ with eigenvalues $+1,-1,-1,+1$ for $c=1,2,3,6$,
respectively. 
Thus, each regularized inner product
$\langle\mathbf H_n,\theta^{\sigma_c}\rangle^{\reg}=\pm2\langle H_n,\eta\rangle^{\reg}$. By Lemma~\ref{lem:pairing}, this is  zero.

In order to apply Bruinier-Ono's \cite[Theorem~6.1]{BO}, we need to compute a quantity known as a Weyl vector. 
By 
\cite[equation~(11)]{BS}, this is 
\[
\rho_{\mathbf H_n,1/c}
 =\frac{\sqrt6}{8\pi}\langle\mathbf H_n,\theta^{\sigma_c}\rangle^{\reg}=0.
\]
Since $6$ is squarefree, this represents all cusps. 

The Borcherds product theorem \cite[Theorem~6.1]{BO} now produces a
weight-zero meromorphic form with unitary character,
\[
\mathcal B_n(\tau)
:=\prod_{m\ge1}(1-q^m)^{\chip(m)h_n(m^2)},
\]
having no cusp zeros or poles. Its divisor is
\begin{equation}\label{eq:CM-divisor}
D_{n,1}+D_{n,11}-D_{n,5}-D_{n,7}.
\end{equation}
For  checking multiplicities, note that the the Heegner cycle for a single component $r$ in \cite[Section~5]{BO} includes both
signs of the corresponding norm vector.
For $-n<-4$, the stabilizer has order two. In the convention of
positive forms used here, that cycle is
$\tfrac12(D_{n,r}+D_{n,-r})$. Summing the four components of $v$ gives
\eqref{eq:CM-divisor}, with multiplicity one.

The rational function $A_n(t)/B_n(t)$ has the same divisor,
and has neither a zero nor a pole at any cusp. It is $1$ at $\infty$
since its numerator and denominator are monic of equal degree.
The quotient $\mathcal B_n/(A_n(t)/B_n(t))$ is nonvanishing and
holomorphic, including at the cusps, with unitary character. The logarithm
of its absolute value descends to a harmonic function on the compact
curve. It is constant, and the open mapping theorem makes the quotient
itself constant. Its leading coefficient at $\infty$ is $1$.
This proves \eqref{eq:product} and also proves that the character is
trivial. Integrality follows form this product formula.
\end{proof}

\section{Canonical congruences and class polynomials}
The main application of these polynomials for us is the following result.
\begin{proposition}\label{prop:W-zero}
Let $\ell\ge5$ be a canonical congruence prime. Suppose that $n$ is squarefree,
$n\equiv23\pmod{24}$, and $\leg{-n}{\ell}=-1$. Then we have the congruence
$W_n:=A_n'B_n-A_nB_n'\equiv0\pmod \ell$.
\end{proposition}
\begin{proof}
By Lemma~\ref{lem:duality}, every exponent
$\chip(m)h_n(m^2)$ in \eqref{eq:product} is divisible by $\ell$ when
$\ell\nmid m$: it is zero if $(m,6)>1$, and otherwise
\eqref{eq:square-vanishing} applies. In characteristic $\ell$, an integral
power $(1-q^m)^{\ell e}$ is equal to $(1-q^{\ell m})^e$. 
The factors with $\ell\mid m$ only involve terms $q^\ell$. Thus,
\[
\overline{A_n(t)/B_n(t)}\in\F_\ell[[q^\ell]],
\]
and its image under $\Theta$ is zero. The chain rule gives
\[
0=\frac{\overline W_n(\bar t)}{\overline B_n(\bar t)^2}\Theta(\bar t).
\]
The series $\Theta(\bar t)$ has leading term $-q^{-1}$, so is nonzero.
Substiting $T\mapsto\bar t=q^{-1}-5+O(q)$ embeds
$\F_\ell(T)$ into $\F_\ell((q))$; the leading power of a nonzero polynomial
cannot cancel. Thus, we have $\overline W_n=0$, as claimed.
\end{proof}


For \(\ell>3\), suppose that the elliptic curve \(E/\overline{\mathbb F}_\ell\) is supersingular, and let \(C\subset E\) be cyclic of order \(6\).
For such a pair, define
\[
\End(E,C):=\{\alpha\in\End(E):\alpha(C)\subseteq C\}.
\]
The ring $\End(E,C)$ is an Eichler order of level $6$ in the quaternion algebra
$B_{\ell,\infty}$ (see, e.g., \cite[Theorem~3.7]{Arpin}).
We also require the norm and trace on $\End(E,C)$:
\[
N(\alpha):=\alpha\alpha^\ast,\quad \quad \mathrm{Tr}(\alpha):=\alpha+\alpha^*,
\]
where $\alpha^*$ denotes its quaternion conjugate.
We require the following estimate.
\begin{lemma}\label{lem:volume}
Let $\ell>3$, let $E/\overline{\F}_\ell$ be supersingular, and let
$C\subset E$ be cyclic of order $6$. If
$\alpha,\beta\in\End(E,C)$ do not commute, then
\begin{equation}\label{eq:volume}
6\ell\le4N(\alpha)N(\beta).
\end{equation}
\end{lemma}
\begin{proof}
In our normalization, $\End(E)$ has covolume $\ell$ for
$\langle x,y\rangle=\mathrm{Tr}(xy^*)$. At $p=2,3$, preserving the
line $C[p]$ imposes an index-$p$ condition on 
$M_2(\Z_p)$, At other primes, there is no additional condition.
Thus, $[\End(E):\End(E,C)]=6$ and $\covol(\End(E,C))=6\ell$.
Noncommutativity makes $1,\alpha,\beta,\alpha\beta$ linearly independent
over $\Q$. Their span $L$ is a full lattice in $\End(E,C)$. Hadamard's inequality,
(see, e.g., \cite[Lemma~2.1.1]{GL}, gives
\[
6\ell\le\covol(L)
\le\sqrt2\cdot\sqrt{2N(\alpha)\cdot 2N(\beta)\cdot
                    2N(\alpha\beta)}
=4N(\alpha)N(\beta).
\]
\end{proof}

The connection to canonical congruence primes is given by the following result.
\begin{proposition}\label{prop:separation}
Let $n$ be squarefree, $n\equiv23\pmod{24}$, and let $\ell>3$ satisfy
$\leg{-n}{\ell}=-1$. If
\begin{equation}\label{eq:collision-bound}
\ell>\frac{(n+1)^2}{24},
\end{equation}
then $\overline A_n$ and $\overline B_n$ are coprime in $\F_\ell[X]$.
\end{proposition}
\begin{proof}
Suppose otherwise. Work in a number field containing all CM coordinates,
at a prime over $\ell$, and then extend so the CM elliptic curves have good
reduction. A common root gives a CM point from each of the two
orientation sets with the same reduced $t$-coordinate.
Lemma~\ref{lem:integral-coordinate} identifies the reductions as the
same pair $(E,C)$ over $\overline{\F}_\ell$. Since $\ell$ is inert in
$K=\Q(\sqrt{-n})$, $E$ is supersingular.
The two CM actions specialize injectively into $\End(E,C)$.

Let $\alpha$ and $\beta$ be the images, under these two actions, of
$\omega:=(1+\sqrt{-n})/2$. Each has trace $1$ and norm $(n+1)/4$.
We now show that they cannot commute. 
For $Q=[a,b,c]$, the complex point is
\[
E_Q=\C/(\Z \tau_Q+\Z),\qquad C_Q=\langle1/6\rangle.
\]
As $\omega=a\tau_Q+(b+1)/2$ and $6|a$, its action on $C_Q$ is the
scalar $(b+1)/2\pmod6$. The possibilities are
\[
\begin{array}{c|cccc}
b\pmod{12}&1&11&5&7\\ \hline
\omega|_{C_Q}\pmod6&1&0&3&4.
\end{array}
\]
These actions persist under specialization, since $\ell\nmid6$.
Hence $\alpha$ acts on $C$ by $0$ or $1$, while $\beta$ acts by
$3$ or $4$.

If $\alpha$ and $\beta$ did commute, the centralizer of $\alpha$ in
$B_{p,\infty}$ would give $\beta\in\Q(\alpha)$.
Their common characteristic polynomial is
$X^2-X+(n+1)/4$, so $\beta=\alpha$ or $\beta=1-\alpha$.
Both possibilities make its action on $C$ a member of $\{0,1\}$,
a contradiction. Lemma~\ref{lem:volume} therefore gives
\[
6\ell\le4\left(\frac{n+1}{4}\right)^2
=\frac{(n+1)^2}{4},
\]
contrary to \eqref{eq:collision-bound}.
\end{proof}
Finally, we obtain the following essential result for our proofs.
\begin{corollary}\label{cor:nonvanishing}
Assume the conditions of Proposition~\ref{prop:separation}. If we also have
$2h(-n)<\ell$, then we have the non-congruence $W_n\not\equiv0\pmod \ell$.
\end{corollary}
\begin{proof}
Suppose $\overline A_n'\overline B_n=\overline A_n\overline B_n'$.
Coprimality gives $\overline A_n\mid\overline A_n'$.
Since $\deg\overline A_n'<\deg\overline A_n$, its derivative must be
zero. In characteristic $\ell$ this forces every nonzero exponent of
$\overline A_n$ to be divisible by $\ell$. But it is monic of degree
$2h(-n)$ strictly between $0$ and $\ell$, a contradiction.
\end{proof}

\section{Proof of Theorem~\ref{thm:main}}

The final key ingredient is a bound of Trevi\~no.
Using the correction in
\cite[Theorem~1]{TrevinoCorrection}, his Theorem~1 implies that if a nonprincipal character modulo
an odd prime $\ell$ is constant on $H$ consecutive integers, then
\begin{equation}\label{eq:character}
H\le3.38\cdot \ell^{1/4}\log \ell.
\end{equation}

This allows us to show the following estimate.
\begin{lemma}\label{lem:small-n}
For every prime $\ell\ge10^{11}$, there is a squarefree integer
$n\equiv23\pmod{24}$ satisfying
\begin{equation}\label{eq:small-n}
\leg{-n}{\ell}=-1,\qquad n+1<\sqrt{24\ell},\qquad 2h(-n)<\ell.
\end{equation}
\end{lemma}
\begin{proof}
Let $k\ge1$ be the first integer with $\leg{1-24k}{\ell}\ne1$.
Such a $k$ exists, since the progression hits every residue modulo
$\ell$. If $c$ represents $24^{-1}\pmod \ell$, then for $1\le j<k$,
\[
\leg{j-c}{\ell}=\leg{-24}{\ell}.
\]
This is an interval of $k-1$ consecutive integers on which a nonprincipal
character is constant. Equation~\eqref{eq:character} gives
$k\le3.38\ell^{1/4}\log \ell+1$. Thus $n=24k-1$ satisfies
\begin{equation}\label{eq:n-bound}
n+1\le81.12\ell^{1/4}\log \ell+24<\sqrt{24\ell}.
\end{equation}
For the final inequality, the ratio of its middle expression to
$\sqrt{24\ell}$ is
\[
\frac{81.12}{\sqrt{24}}\frac{\log \ell}{\ell^{1/4}}
       +\frac{24}{\sqrt{24\ell}}.
\]
Both summands decrease for $\ell>e^4$. At $\ell=10^{11}$, the elementary
bounds $\log(10^{11})<26$, $(10^{11})^{1/4}>560$, and
$\sqrt{24}>4.8$ give an upper bound
\[
\frac{81.12\cdot26}{4.8\cdot560}+\frac{24}{10^6}<0.785<1.
\]
This proves \eqref{eq:n-bound} throughout the required range.

In particular, $0<n<\ell$. Thus, the first character value different from
$1$ cannot be zero; it equals $-1$. If $n$ had a square factor $u^2>1$,
then $(u,6\ell)=1$ and $u^2\equiv1\pmod{24}$. The smaller positive integer
$n/u^2$ would still be $23\pmod{24}$ and would have the same inert
character, contradicting the minimality of $k$. Therefore $n$ is
squarefree.

Finally, a reduced positive form of discriminant $-n$ has
$1\le a\le M:=\lfloor\sqrt{n/3}\rfloor$ and $|b|\le a$.
Counting possible pairs $(a,b)$, without imposing all the reduction
conditions, gives
\[
h(-n)\le\sum_{a=1}^M(2a+1)=M^2+2M
\le\frac n3+2\sqrt{\frac n3}<n\qquad(n\ge23).
\]
Since $\ell\ge10^{11}>96$ and $n<\sqrt{24\ell}$, this implies
$2h(-n)<2n<2\sqrt{24\ell}<\ell$.
\end{proof}

\begin{proof}[Proof of Theorem~\ref{thm:main}]
The congruences at $5,7,13$ are known, and \cite[Theorem~3]{AK} excludes
all other primes below $10^{11}$. Suppose that a prime $\ell\ge10^{11}$
were a canonical congruence prime. Choose $n$ using
Lemma~\ref{lem:small-n}. Proposition~\ref{prop:W-zero} shows that
$W_n\equiv0\pmod \ell$, while Corollary~\ref{cor:nonvanishing} shows 
$W_n\not\equiv0\pmod \ell$. This contradiction completes the proof.
\end{proof}


\end{document}